\documentclass[10pt]{amsart}
\usepackage{amsfonts}
\usepackage{ifthen}
\usepackage{amsthm}
\usepackage{amsmath}
\usepackage{graphicx}
\usepackage{amscd,amssymb,amsthm}
\usepackage{enumerate}

\newcounter{minutes}
\divide\time by 60
\newcounter{hours}
\multiply\time by 60 \addtocounter{minutes}{-\time}

\title{Tur\'an type inequalities for Struve functions}

\author[\'A. Baricz]{\'Arp\'ad Baricz}
\address{Department of Economics,  Babe\c{s}-Bolyai University, Cluj-Napoca 400591, Romania}
\address{Institute of Applied Mathematics, John von Neumann Faculty of Informatics, \'Obuda University, 1034 Budapest, Hungary}
\email{bariczocsi@yahoo.com}
\author[S. Ponnusamy]{Saminathan Ponnusamy}
\address{Indian Statistical Institute, Chennai Centre, Society for Electronic Transactions and Security,
MGR Knowledge City, CIT Campus, Taramani, Chennai 600113, India}
\email{samy@iitm.ac.in}
\author[S. Singh]{Sanjeev Singh}
\address{Department of Mathematics,
Indian Institute of Technology Madras, Chennai 600036, India}
\email{sanjeevsinghiitm@gmail.com}

\newtheorem{theorem}{Theorem}

\newtheorem{lemma}{Lemma}

\begin{document}

\def\thefootnote{}
\footnotetext{ \texttt{File:~\jobname .tex,
          printed: \number\year-0\number\month-0\number\day,
          \thehours.\ifnum\theminutes<10{0}\fi\theminutes}
} \makeatletter\def\thefootnote{\@arabic\c@footnote}\makeatother



\maketitle


\begin{abstract}
Some Tur\'an type inequalities for Struve functions of the first kind are deduced by using various methods developed
in the case of Bessel functions of the first and second kind. New formulas, like Mittag-Leffler expansion,
infinite product representation for Struve functions of the first kind, are obtained, which may be of independent interest.
Moreover, some complete monotonicity results and functional inequalities are deduced for Struve functions of the second kind.
These results complement naturally the known results for a particular case of Lommel functions of the
first kind, and for modified Struve functions of the first and second kind.
\end{abstract}

\section{\bf Tur\'an type inequalities for Struve functions of the first kind}
\setcounter{equation}{0}

Let us start with a well-known relation between Bessel functions of the first kind $J_{\nu}$ and
Struve functions of the first kind $\mathbf{H}_{\nu}.$ Namely, for all $n\in\{0,1,\dots\}$ and
$x\in\mathbb{R}$ we have \cite[p. 291]{nist}
$$\mathbf{H}_{-n-\frac{1}{2}}(x)=(-1)^nJ_{n+\frac{1}{2}}(x).
$$
Now, let us recall the Tur\'an type inequality for Bessel functions of the first kind, that is,
\begin{equation}\label{turan1}
J_{\nu}^2(x)-J_{\nu-1}(x)J_{\nu+1}(x)\geq 0,
\end{equation}
where $x\in\mathbb{R}$ and $\nu>-1.$ Combining the above relation with \eqref{turan1} we obtain the
Tur\'an type inequality
\begin{equation}\label{turan2}
\mathbf{H}_{\nu}^2(x)-\mathbf{H}_{\nu-1}(x)\mathbf{H}_{\nu+1}(x)\geq0,
\end{equation}
which holds for all $x\in\mathbb{R}$ and $\nu\in\left\{-\frac{1}{2},-\frac{3}{2},\dots\right\}.$
Moreover, since for $x\in\mathbb{R}$ and $\nu\geq 0$ the Tur\'an type inequality \eqref{turan1} can be improved as
$$J_{\nu}^2(x)-J_{\nu-1}(x)J_{\nu+1}(x)\geq\frac{1}{\nu+1}J_{\nu}^2(x),
$$
it follows that the inequality \eqref{turan2} can be improved too as
\begin{equation}\label{turan4}
\mathbf{H}_{\nu}^2(x)-\mathbf{H}_{\nu-1}(x)\mathbf{H}_{\nu+1}(x)\geq\frac{1}{1-\nu}\mathbf{H}_{\nu}^2(x),
\end{equation}
which holds for all $x\in\mathbb{R}$ and $\nu\in\left\{-\frac{1}{2},-\frac{3}{2},\dots\right\}.$
For more details on the above Tur\'an type inequalities for Bessel functions of the first kind
we refer to the papers \cite{skov,szasz1,szasz2,thiru} and also to the survey paper \cite{bp1}.
Taking into account the above inequalities it is natural to ask whether the Tur\'an type
inequalities \eqref{turan2} and/or \eqref{turan4} hold true for other values of $\nu.$
In this paper we will concentrate on this problem and we present some interesting results concerning
Tur\'an type inequalities for Struve functions of the first and second kind. As we can see below the
analysis of Struve functions is somewhat more complicated than that of Bessel functions, however, its nature
is quite similar for some values of $\nu.$ This section is devoted to Tur\'an type inequalities for
Struve functions of the first kind, while the next section contains some results, like Tur\'an type
inequalities and complete monotonicity results on Struve functions of the second kind. Before we present
the main results of this section we first show some preliminary results which will be used in the sequel
and which may be of independent interest. Since Struve functions are frequently used in many places in
physics and applied mathematics, we believe that our results may be useful for other scientists
interested in Struve functions. We also note that the analogous results for modified Struve functions
of the first and second kind were already deduced by Baricz and Pog\'any \cite{bp2,bp3} by using the
techniques developed in the case of modified Bessel functions of the first and second kind.
Moreover, the results presented in this section complement naturally the known results for a particular
case of Lommel functions of the first kind, obtained recently by Baricz and Koumandos \cite{bk}.

The next result is analogous to the well-known result for Bessel functions of the first kind.

\begin{lemma}
If $|\nu|\leq \frac{1}{2},$ then the Hadamard factorization of the real entire function
$\mathcal{H}_{\nu}:\mathbb{R}\to(-\infty,1],$ defined by
$\mathcal{H}_{\nu}(x)=\sqrt{\pi}2^{\nu}x^{-\nu-1}\Gamma\left(\nu+\frac{3}{2}\right)\mathbf{H}_{\nu}(x),$ reads as follows
\begin{equation}\label{product}
\mathcal{H}_{\nu}(x)=\prod_{n\geq 1}\left(1-\frac{x^2}{h_{\nu,n}^2}\right),
\end{equation}
where $h_{\nu,n}$ stands for the $n$th positive zero of the Struve function $\mathbf{H}_{\nu}.$
The above infinite product is absolutely convergent and if $|\nu|\leq \frac{1}{2}$ and
$x\neq h_{\nu,n},$ $n\in\{1,2,\dots\},$ then the Mittag-Leffler expansion of the Struve function
$\mathbf{H}_{\nu}$ is as follows
\begin{equation}\label{mittag}
\frac{\mathbf{H}_{\nu-1}(x)}{\mathbf{H}_{\nu}(x)}=\frac{2\nu+1}{x}+\sum_{n\geq1}\frac{2x}{x^2-h_{\nu,n}^2}.
\end{equation}
\end{lemma}

\begin{proof}[\bf Proof]
By using the power series expansion of the Struve function $\mathbf{H}_{\nu}$ \cite[p. 288]{nist}
$$\mathbf{H}_{\nu}(x)=\left(\frac{x}{2}\right)^{\nu+1}\sum_{n\geq0}\frac{(-1)^n\left(\frac{x}{2}\right)^{2n}}
{\Gamma\left(n+\frac{3}{2}\right)\Gamma\left(n+\nu+\frac{3}{2}\right)}
$$
we obtain that
$$\mathcal{H}_{\nu}(x)=\frac{\sqrt{\pi}}{2}\sum_{n\geq0}\frac{(-1)^n\Gamma\left(\nu+\frac{3}{2}\right)x^{2n}}
{2^{2n}\Gamma\left(n+\frac{3}{2}\right)\Gamma\left(n+\nu+\frac{3}{2}\right)}.
$$
Taking into consideration the well-known limits
$$
\lim_{n\to \infty} \frac{\log \Gamma(n+c)}{n\,\log n}=1,\quad
\lim_{n\to \infty} \frac{[\Gamma(n+c)]^{1/n}}{n}=\frac{1}{e},
$$
where $c$ is a positive constant, and \cite[p. 6, Theorems 2 and 3]{lev}, we infer that the entire
function $\mathcal{H}_{\nu}$ is of growth order $\rho=\frac{1}{2}$ and of exponential type $\sigma=0$. Namely,
for $\nu>-\frac{3}{2}$ we have
$$\rho=\lim_{n\to\infty}\frac{n\log n}{2n\log2 +\log\Gamma\left(n+\frac{3}{2}\right)+\log\Gamma\left(n+\nu+\frac{3}{2}\right)}
=\frac{1}{2}
$$
and
$$\sigma=\frac{1}{\rho e}\lim_{n\to\infty}\frac{n}{\sqrt[n]{2^{2n}\Gamma\left(n+\frac{3}{2}\right)\Gamma\left(n+\nu+\frac{3}{2}\right)}}=0.
$$
Now, recall that according to Steinig \cite[p. 367]{steinig} if $|\nu|<\frac{1}{2},$ then all zeros
$h_{\nu,n}$ of the Struve function $\mathbf{H}_{\nu}$ are real and simple. Moreover, since \cite[p. 291]{nist}
$$\mathbf{H}_{-\frac{1}{2}}(x)=\sqrt{\frac{2}{\pi x}}\sin x, \ \ \ \mathbf{H}_{\frac{1}{2}}(x)=\sqrt{\frac{2}{\pi x}}(1-\cos x),
$$
it is clear that all zeros of $\mathbf{H}_{-\frac{1}{2}}$ and $\mathbf{H}_{\frac{1}{2}}$ are real and simple.
With this the rest of the proof of \eqref{product} follows by applying Hadamard's Theorem \cite[p. 26]{lev}.
Now, since the infinite product in \eqref{product} is absolutely convergent, by taking the logarithm of both
sides of \eqref{product} and then differentiating we obtain
\begin{equation}\label{logder}
\frac{x\mathbf{H}_{\nu}'(x)}{\mathbf{H}_{\nu}(x)}=\nu+1+\sum_{n\geq 1}\frac{2x^2}{x^2-h_{\nu,n}^2},
\end{equation}
where $|\nu|\leq\frac{1}{2}$ and $x\neq h_{\nu,n},$ $n\in\{1,2,\dots\}.$ The rest of the
proof of \eqref{mittag} follows from \eqref{logder} and
\begin{equation}\label{rec1}
\mathbf{H}_{\nu-1}(x)=\frac{\nu}{x}\mathbf{H}_{\nu}(x)+\mathbf{H}_{\nu}'(x)
\end{equation}
which is obtained from the relations \cite[p. 292]{nist}
\begin{equation}\label{rec2}
\mathbf{H}_{\nu-1}(x)+\mathbf{H}_{\nu+1}(x)=\frac{2\nu}{x}\mathbf{H}_{\nu}(x)+
\frac{\left(\frac{x}{2}\right)^{\nu}}{\sqrt{\pi}\Gamma\left(\nu+\frac{3}{2}\right)}
\end{equation}
and
\begin{equation}\label{rec3}
\mathbf{H}_{\nu-1}(x)-\mathbf{H}_{\nu+1}(x)=2\mathbf{H}_{\nu}'(x)-
\frac{\left(\frac{x}{2}\right)^{\nu}}{\sqrt{\pi}\Gamma\left(\nu+\frac{3}{2}\right)}.
\end{equation}
\end{proof}

Now, we are ready to state our main result of this section.

\begin{theorem}\label{th1}
The following assertions are valid:
\begin{enumerate}
\item[\bf a.] If $\nu\in\left[-\frac{3}{2},-\frac{1}{2}\right]$ and $x\in\mathbb{R},$ then the Tur\'an
type inequality \eqref{turan2} holds true.
\item[\bf b.] If $|\nu|\leq\frac{1}{2}$ and $|x|\leq h_{\nu,1},$ then the Tur\'an type inequality \eqref{turan2} holds true.
\item[\bf c.] If $\nu\in\left[-\frac{3}{2},-\frac{1}{2}\right]$ and $x\in\mathbb{R},$ then
\begin{equation}\label{turanlag}
\mathbf{H}_{\nu}^2(x)-\mathbf{H}_{\nu-1}(x)\mathbf{H}_{\nu+1}(x)\geq\frac{1}{x}\mathbf{H}_{\nu}(x)\mathbf{H}_{\nu+1}(x).
\end{equation}
Moreover, if $\nu\in\left(-\frac{3}{2},-\frac{1}{2}\right]$ and $|x|< h_{\nu+1,1},$ then the next Tur\'an type inequality holds
\begin{equation}\label{turannew}
\mathbf{H}_{\nu}^2(x)-\mathbf{H}_{\nu-1}(x)\mathbf{H}_{\nu+1}(x)\geq\frac{1}{2\nu+3}\mathbf{H}_{\nu}^2(x).
\end{equation}
\item[\bf d.] If $\nu\geq \frac{3}{2}$ and $|x|\leq \pi,$ then the Tur\'an type inequality \eqref{turan2} is valid.
\item[\bf e.] If $\nu>\frac{3}{2}$ and $|x|<\pi,$ then the counterpart of the Tur\'an type inequality \eqref{turan2} is as follows:
\begin{equation}\label{T2}
\mathbf{H}_{\nu}^2(x)-\mathbf{H}_{\nu-1}(x)\mathbf{H}_{\nu+1}(x)\leq\frac{1}{\nu+\frac{1}{2}}\mathbf{H}_{\nu}^2(x).
\end{equation}
\end{enumerate}
\end{theorem}

\begin{proof}[\bf Proof]
{\bf a.} Let us consider the notation
$$\mathbf{\Delta}_{\nu}(x)=\mathbf{H}_{\nu}^2(x)-\mathbf{H}_{\nu-1}(x)\mathbf{H}_{\nu+1}(x).
$$
By using the recurrence relation \eqref{rec1} for $\nu$ and $\nu-1,$ and the Mittag-Leffler expansion we obtain that
$$\frac{\mathbf{\Delta}_{\nu-1}(x)}{\mathbf{H}_{\nu}^2(x)}=
\frac{1}{x}\frac{\mathbf{H}_{\nu-1}(x)}{\mathbf{H}_{\nu}(x)}
-\left[\frac{\mathbf{H}_{\nu-1}(x)}{\mathbf{H}_{\nu}(x)}\right]'=
\frac{2(2\nu+1)}{x^2}+\sum_{n\geq 1}\frac{4x^2}{(x^2-h_{\nu,n}^2)^2}\geq0
$$
for all $|\nu|\leq\frac{1}{2}$ and $x>0,$ $x\neq h_{\nu,n},$ $n\in\{1,2,\dots\}.$ Since for each
$n\in\{1,2,\dots\}$ we have
$$\mathbf{\Delta}_{\nu-1}(h_{\nu,n})=\mathbf{H}_{\nu-1}^2(h_{\nu,n})>0,
$$
it follows that $\mathbf{\Delta}_{\nu-1}(x)\geq0$ for $x>0.$ Moreover, the expression $\mathbf{\Delta}_{\nu-1}(x)$
is even in $x,$ and thus the above Tur\'an type inequality is valid for all $x\in\mathbb{R}$ and $|\nu|\leq \frac{1}{2}.$
Now, changing $\nu$ to $\nu+1,$ we obtain that indeed if $\nu\in\left[-\frac{3}{2},-\frac{1}{2}\right]$ and
$x\in\mathbb{R},$ then the Tur\'an type inequality \eqref{turan2} holds true.

{\bf b.} Combining the recurrence relations
\eqref{rec2} and \eqref{rec3} we obtain that
\begin{equation}\label{rec4}
\mathbf{H}_{\nu+1}(x)=\frac{\nu}{x}\mathbf{H}_{\nu}(x)-\mathbf{H}_{\nu}'(x)
+\frac{\left(\frac{x}{2}\right)^{\nu}}{\sqrt{\pi}\Gamma\left(\nu+\frac{3}{2}\right)}.
\end{equation}
Now, combining \eqref{rec1} and \eqref{rec4} it follows that
$$\mathbf{\Delta}_{\nu}(x)=\left(1-\frac{\nu^2}{x^2}\right)\mathbf{H}_{\nu}^2(x)+\left[\mathbf{H}_{\nu}'(x)\right]^2-
\frac{\left(\frac{x}{2}\right)^{\nu}}{\sqrt{\pi}\Gamma\left(\nu+\frac{3}{2}\right)}\mathbf{H}_{\nu-1}(x).
$$
On the other hand, the Struve function is the particular solution of the Struve differential equation \cite[p. 288]{nist}
and consequently we have
$$\mathbf{H}_{\nu}''(x)+\frac{1}{x}\mathbf{H}_{\nu}'(x)+\left(1-\frac{\nu^2}{x^2}\right)\mathbf{H}_{\nu}(x)=
\frac{\left(\frac{x}{2}\right)^{\nu-1}}{\sqrt{\pi}\Gamma\left(\nu+\frac{1}{2}\right)},
$$
which implies that
\begin{equation}\label{delta}
\mathbf{\Delta}_{\nu}(x)=\frac{\left(\frac{x}{2}\right)^{\nu-1}}{\sqrt{\pi}\Gamma\left(\nu+\frac{1}{2}\right)}
\left[\mathbf{H}_{\nu}(x)-\frac{x}{2\nu+1}\mathbf{H}_{\nu-1}(x)\right]-
\frac{1}{x}\mathbf{H}_{\nu}^2(x)\left[\frac{x\mathbf{H}_{\nu}'(x)}{\mathbf{H}_{\nu}(x)}\right]'.
\end{equation}
Differentiating both sides of \eqref{logder} we obtain
$$\left[\frac{x\mathbf{H}_{\nu}'(x)}{\mathbf{H}_{\nu}(x)}\right]'=-\sum_{n\geq1}\frac{4xh_{\nu,n}^2}{(x^2-h_{\nu,n}^2)^2}<0
$$
for all $x>0,$ $x\neq h_{\nu,n},$ $n\in\{1,2,\dots\}$ and $|\nu|\leq\frac{1}{2}.$ Moreover,
by using the Mittag-Leffler expansion \eqref{mittag} we obtain that
\begin{equation}\label{ineqquo}
\frac{x\mathbf{H}_{\nu-1}(x)}{\mathbf{H}_{\nu}(x)}<2\nu+1
\end{equation}
for all $|\nu|\leq \frac{1}{2}$ and $x\in(0,h_{\nu,1}).$ Finally, by using the above inequalities
together with \eqref{delta} we conclude that \eqref{turan2} is valid for $|\nu|\leq \frac{1}{2}$ and $x\in(0,h_{\nu,1}).$
Since $\mathbf{\Delta}_{\nu}(0)=0,$ in \eqref{turan2} we have equality when $x=0.$ Moreover, if $|\nu|\leq1/2,$
then $\mathbf{\Delta}_{\nu}(h_{\nu,1})=-\mathbf{H}_{\nu-1}(h_{\nu,1})\mathbf{H}_{\nu+1}(h_{\nu,1})>0,$
since the smallest positive zero of $\mathbf{H}_{\nu-1}$ is nearer the origin than that of $\mathbf{H}_{\nu},$
and the positive zeros of $\mathbf{H}_{\nu}$ and $\mathbf{H}_{\nu-1}$ separate each other, according
to \cite[p. 373]{steinig}. Now, since the expression $\mathbf{\Delta}_{\nu}(x)$ is even in $x,$ these in
turn imply that \eqref{turan2} holds true for $|x|\leq h_{\nu,1}.$

{\bf c.} By using \eqref{rec1} for $\nu$ and $\nu+1$ we obtain
$$\frac{\mathbf{\Delta}_{\nu}(x)}{\mathbf{H}_{\nu}^2(x)}=\frac{1}{x}\frac{\mathbf{H}_{\nu+1}(x)}{\mathbf{H}_{\nu}(x)}+
\left[\frac{\mathbf{H}_{\nu+1}(x)}{\mathbf{H}_{\nu}(x)}\right]'.
$$
On the other hand, the Mittag-Leffler expansion \eqref{mittag} for $\nu+1$ instead of $\nu$ implies that
$$\left[\frac{\mathbf{H}_{\nu}(x)}{\mathbf{H}_{\nu+1}(x)}\right]'=-\frac{2\nu+3}{x^2}
+\sum_{n\geq1}\frac{-2(x^2+h_{\nu+1,n}^2)}{(x^2-h_{\nu+1,n}^2)^2}\leq0
$$
for all $\nu\in\left[-\frac{3}{2},-\frac{1}{2}\right],$ and $x\neq h_{\nu+1,n},$ $n\in\{1,2,\dots\}.$ This implies
that the Tur\'an type inequality \eqref{turanlag} is valid for $x\in\mathbb{R}$ and
$\nu\in\left[-\frac{3}{2},-\frac{1}{2}\right].$ By using \eqref{ineqquo} for $\nu+1$ instead of $\nu,$
we obtain that \eqref{turannew} holds for $x\in(0,h_{\nu+1,1})$ and $\nu\in\left(-\frac{3}{2},-\frac{1}{2}\right],$
and hence \eqref{turannew} is valid for $|x|< h_{\nu+1,1}$ and $\nu\in\left(-\frac{3}{2},-\frac{1}{2}\right].$

{\bf d.} By using \eqref{rec1}, \eqref{rec3} and \eqref{rec4} we obtain
$$\mathbf{\Delta}'_{\nu}(x)=\frac{2}{x}\mathbf{H}_{\nu-1}(x)\mathbf{H}_{\nu+1}(x)+
\frac{\left(\frac{x}{2}\right)^{\nu-1}}{\sqrt{\pi}\Gamma\left(\nu+\frac{1}{2}\right)}
\left[\frac{x}{2\nu+1}\mathbf{H}_{\nu}(x)-\mathbf{H}_{\nu+1}(x)\right].
$$
In view of the integral representation of $\mathbf{H}_{\nu}$ \cite[p. 292]{nist}
\begin{equation}\label{integral}
\mathbf{H}_{\nu}(x)=\frac{2\left(\frac{x}{2}\right)^{\nu}}{\sqrt{\pi}\Gamma\left(\nu+\frac{1}{2}\right)}
\int_0^1(1-t^2)^{\nu-\frac{1}{2}}\sin(xt)dt,
\end{equation}
we get that
$$\frac{x}{2\nu+1}\mathbf{H}_{\nu}(x)-\mathbf{H}_{\nu+1}(x)
=\frac{2\left(\frac{x}{2}\right)^{\nu+1}}{\sqrt{\pi}\Gamma\left(\nu+\frac{3}{2}\right)}
\int_0^1t^2(1-t^2)^{\nu-\frac{1}{2}}\sin(xt)dt,
$$
which is nonnegative if $x\in[0,\pi].$ On the other hand, it is known \cite[p. 291]{nist} that
$\mathbf{H}_{\nu}(x)\geq0$ if $x>0$ and $\nu\geq \frac{1}{2},$ and combining this with the above result,
we obtain that $\mathbf{\Delta}_{\nu}'(x)\geq0$ when $x\in[0,\pi]$ and $\nu\geq\frac{3}{2}.$ This implies that
$\mathbf{\Delta}_{\nu}(x)\geq 0$ when $x\in[0,\pi]$ and $\nu\geq\frac{3}{2},$ and using again the fact that
$\mathbf{\Delta}_{\nu}(x)$ is even in $x,$ we conclude that \eqref{turan2} holds for $|x|\leq\pi$ and $\nu\geq\frac{3}{2}.$

{\bf e.} We define the function $\mathbb{H}_{\nu}:\mathbb{R}\to\mathbb{R}$ by
$\mathbb{H}_{\nu}(x)=2^{\nu}x^{-\nu}\Gamma\left(\nu+\frac{1}{2}\right)\mathbf{H}_{\nu}(x).$ In view
of \eqref{integral} this function may be represented as
$$\mathbb{H}_{\nu}(x)=\frac{2}{\sqrt{\pi}} \int_0^{1}(1-t^2)^{\nu-\frac{1}{2}}\sin(xt)dt.
$$
Since the above integrand is log-convex in $\nu$ when $x\in(0,\pi)$ and the integral preserves the
log-convexity, it follows that $\nu\mapsto \mathbb{H}_{\nu}(x)$ is log-convex on $\left(\frac{1}{2},\infty\right)$
for $x\in(0,\pi)$ fixed. Here we used tacitly the inequality \cite[p. 291]{nist} $\mathbf{H}_{\nu}(x)>0,$ which
holds for $\nu>\frac{1}{2}$ and $x>0.$ Thus, for all $\nu_{1},\nu_{2}>\frac{1}{2},$ $\alpha\in[0,1]$ and $x>0$ we have
$$\mathbb{H}_{\alpha\nu_{1}+(1-\alpha)\nu_{2}}(x)\leq \left[\mathbb{H}_{\nu_{1}}(x)\right]^{\alpha}
\left[\mathbb{H}_{\nu_{2}}(x)\right]^{1-\alpha}.
$$
Choosing $\nu_{1}=\nu-1,$ $\nu_{2}=\nu+1,$ $\alpha=\frac{1}{2}$, the above inequality reduces to the Tur\'an type inequality
$$\mathbb{H}_{\nu}^2(x)-\mathbb{H}_{\nu-1}(x)\mathbb{H}_{\nu+1}(x)<0,
$$
which is valid for $\nu>\frac{3}{2}$ and $x>0,$ and this is equivalent to the Tur\'an type inequality \eqref{T2}.
\end{proof}

\vskip2mm
\subsection*{\bf Concluding remarks and further results}
\vskip2mm

{\bf A.} We note that the right-hand side of \eqref{turan4} is negative when $\nu>1,$ and thus the
Tur\'an type inequality \eqref{turan4} is interesting only when $\nu<1.$ Now, if one looks at the right-hand
sides of \eqref{turan4} and \eqref{turannew} it is natural to ask what is best constant $\alpha_{\nu}$
depending on $\nu$ and not depending on $x$ for which we have the Tur\'an type inequality
$$\mathbf{H}_{\nu}^2(x)-\mathbf{H}_{\nu-1}(x)\mathbf{H}_{\nu+1}(x)\geq\alpha_{\nu}\mathbf{H}_{\nu}^2(x).
$$
Since close to the origin the Struve function behaves as a simple power, that is, as $x\to 0$ we have
$$\mathbf{H}_{\nu}(x)\sim \frac{x^{\nu+1}}{2^{\nu}\sqrt{\pi}\Gamma\left(\nu+\frac{3}{2}\right)},
$$
it follows that as $x\to0$ we get
$$\frac{\mathbf{\Delta}_{\nu}(x)}{\mathbf{H}_{\nu}^2(x)}\sim \frac{1}{\nu+\frac{3}{2}}.
$$
This implies that $\alpha_{\nu}=1/\left(\nu+\frac{3}{2}\right),$ and then \eqref{turannew} can be improved.
Now, as the argument approaches infinity, the Struve function generally behaves either as a power or as a
damped sinusoid, that is, for $\nu>\frac{1}{2}$ we have
$$\mathbf{H}_{\nu}(x)\sim \frac{x^{\nu-1}}{2^{\nu-1}\sqrt{\pi}\Gamma\left(\nu+\frac{1}{2}\right)},
$$
which implies that as $x\to\infty$ and $\nu>\frac{3}{2}$,
$$\frac{\mathbf{\Delta}_{\nu}(x)}{\mathbf{H}_{\nu}^2(x)}\sim \frac{1}{\nu+\frac{1}{2}}.
$$
This shows that the constant $1/\left(\nu+\frac{1}{2}\right)$ in \eqref{T2} is optimal, and cannot be
improved by using any other constant depending only on $\nu.$ Thus, in this sense the Tur\'an type
inequality \eqref{T2} is sharp.

{\bf B.} It is worth also to mention that by using the formula \cite[p. 292]{nist}
$$\mathbf{H}_{\nu}(x)=\left(\frac{x}{2\pi}\right)^{\frac{1}{2}}
\sum_{n\geq0}\frac{\left(\frac{x}{2}\right)^{n}}{n!\left(n+\frac{1}{2}\right)}J_{n+\nu+\frac{1}{2}}(x),
$$
we obtain that
$$\mathbf{\Delta}_{\nu}(x)=\frac{x}{2\pi}\sum_{n\geq0}
\sum_{m=0}^n\frac{\left(\frac{x}{2}\right)^{n}}{n!(n-m)!\left(m+\frac{1}{2}\right)\left(n-m+\frac{1}{2}\right)}\cdot {}_J\Delta_{\nu}(x),
$$
where
$${}_J\Delta_{\nu}(x)=J_{m+\nu+\frac{1}{2}}(x)J_{n-m+\nu+\frac{1}{2}}(x)-J_{m+\nu-\frac{1}{2}}(x)J_{n-m+\nu+\frac{3}{2}}(x).
$$
Now, by using Nicholson's formula \cite[p. 225]{nist}
$$J_{\mu}(z)J_{\nu}(z)=\frac{2}{\pi}\int_0^{\frac{\pi}{2}}J_{\mu+\nu}(2z\cos\theta)\cos(\mu-\nu)d\theta
$$
we obtain that
\begin{align*}
{}_J\Delta_{\nu}(x)=\frac{2}{\pi}\int_0^{\frac{\pi}{2}}J_{n+2\nu+1}(2x\cos\theta)
\left[\cos(k\theta)-\cos((k+2)\theta)\right]d\theta=\frac{1}{\pi x}\int_0^{2x}J_{n+2\nu+1}(u)\Phi(u)du,
\end{align*}
where
$$\Phi(u)=\Phi(2x\cos\theta)=\frac{\cos(k\theta)-\cos((k+2)\theta)}{\sin\theta}
$$
and $k=n-2m.$ Thus, to verify the Tur\'an type inequality $\mathbf{\Delta}_{\nu}(x)\geq0$ it would be
enough to show the positivity of the integral $\int_0^{2x}J_{n+2\nu+1}(u)\Phi(u)du.$ On the other hand,
according to \cite[Lemma 2.1]{ismail}, if the function $\varphi$ is positive non-increasing and continuous
for $0<t<x$, then for $\nu>-1$ and $x>0$ we have $\int_0^xJ_{\nu}(t)\varphi(t)dt>0.$ Consequently, it would
be enough to show that $\Phi$ is positive and non-increasing. However, we were unable to show this for each
$m\in\{0,1,\dots,n\}$ and $n\in\{0,1,\dots\}.$

{\bf C.} We would like to note also  that the infinite product representation \eqref{product} and the
Mittag-Leffler expansion \eqref{mittag} may be useful also to deduce other inequalities for the Struve
function $\mathbf{H}_{\nu}.$ For example, we can obtain some lower and upper bounds for $\mathbf{H}_{\nu}$
in terms of the Bessel function of the first kind $J_{\nu}.$ According to Steinig \cite[p. 367]{steinig}
for all $n\in\{1,2,\dots\}$ and $|\nu|<\frac{1}{2}$ we have that $j_{\nu,n}<h_{\nu,n}<j_{\nu,n+1},$
where $j_{\nu,n}$ stands for the $n$th positive zero of the Bessel function $J_{\nu}.$ By using these
inequalities we obtain for $|\nu|<\frac{1}{2}$ and $x\in(0,j_{\nu,1})$
$$\prod_{n\geq 1}\left(1-\frac{x^2}{j_{\nu,n}^2}\right)<\prod_{n\geq 1}\left(1-\frac{x^2}{h_{\nu,n}^2}\right)
<\prod_{n\geq 1}\left(1-\frac{x^2}{j_{\nu,n+1}^2}\right),
$$
which in turn implies that
$$\frac{\Gamma(\nu+1)}{\Gamma\left(\nu+\frac{3}{2}\right)}xJ_{\nu}(x)<\mathbf{H}_{\nu}(x)<
\frac{\Gamma(\nu+1)}{\Gamma\left(\nu+\frac{3}{2}\right)}\frac{j_{\nu,1}^2}{j_{\nu,1}^2-x^2}xJ_{\nu}(x).
$$
Here we used the infinite product representation of the Bessel function of the first kind \cite[p. 235]{nist}
$$J_{\nu}(x)=\frac{\left(\frac{x}{2}\right)^{\nu}}{\Gamma(\nu+1)}\prod_{n\geq 1}\left(1-\frac{x^2}{j_{\nu,n}^2}\right).
$$
Another example is an improvement of the inequality \eqref{ineqquo}. Namely, by using again the
inequalities $h_{\nu,n}<j_{\nu,n+1},$ $n\in\{1,2,\dots\},$ we obtain for $x\in(0,h_{\nu,1})$ and
$|\nu|\leq\frac{1}{2}$ the inequality
$$\sum_{n\geq 1}\frac{2}{x^2-h_{\nu,n}^2}<-\sum_{n\geq 1}\frac{2}{h_{\nu,n}^2}<-\sum_{n\geq 1}\frac{2}{j_{\nu,n+1}^2}
=\frac{2}{j_{\nu,1}^2}-\sum_{n\geq 1}\frac{2}{j_{\nu,n}^2}=\frac{2}{j_{\nu,1}^2}-\frac{1}{2(\nu+1)},
$$
which implies that
$$\frac{x\mathbf{H}_{\nu-1}(x)}{\mathbf{H}_{\nu}(x)}<2\nu+1+\left(\frac{2}{j_{\nu,1}^2}-\frac{1}{2(\nu+1)}\right)x^2.
$$
Note that this is indeed an improvement of \eqref{ineqquo} since \cite[eq. 6.7]{muldoon} $j_{\nu,1}^2>4(\nu+1)$ for $\nu>-1.$

{\bf D.} We mention that the Tur\'an type inequality \eqref{turanlag} can be deduced from a
Laguerre type inequality. More precisely, by using \eqref{product} and \eqref{rec1} we obtain that
$$\left[\frac{\mathcal{H}_{\nu}'(x)}{\mathcal{H}_{\nu}(x)}\right]'=\frac{\nu+1}{x^2}
+\left[\frac{\mathbf{H}_{\nu}'(x)}{\mathbf{H}_{\nu}(x)}\right]'=
\frac{2\nu+1}{x^2}+\left[\frac{\mathbf{H}_{\nu-1}(x)}{\mathbf{H}_{\nu}(x)}\right]'
=-\sum_{n\geq1}\frac{2(h_{\nu,n}^2+x^2)}{(h_{\nu,n}^2-x^2)^2},
$$
which implies that the next Laguerre type inequality is valid
\begin{equation}\label{laguerre}
\left[\mathcal{H}_{\nu}'(x)\right]^2-\mathcal{H}_{\nu}(x)\mathcal{H}_{\nu}''(x)\geq0
\end{equation}
for all $|\nu|\leq \frac{1}{2}$ and $x\in\mathbb{R}.$ But, this inequality is equivalent to
$$\mathbf{\Delta}_{\nu-1}(x)\geq \frac{2\nu+1}{x^2}\mathbf{H}_{\nu}^2(x)+\frac{1}{x}\mathbf{H}_{\nu}(x)\mathbf{H}_{\nu-1}(x),
$$
which implies \eqref{turanlag}. Note that \eqref{laguerre} is in fact a particular case of a more general
inequality. For this recall that the real entire function $\phi,$
defined by
$$\phi(z)=\varphi(z;t)=\sum_{n\geq0}b_n(t)\frac{z^n}{n!},
$$
is said to be in the Laguerre-P\'olya class, if $\phi(z)$ can be expressed in the form
$$\phi(z)=cz^de^{-\alpha z^2+\beta z}\prod_{n=1}^{\omega}\left(1-\frac{z}{z_n}\right)e^{\frac{z}{z_n}},\ \ \ 0\leq\omega\leq\infty,
$$
where $c$ and $\beta$ are real, $z_n$'s are real and nonzero for all
$n\in\{1,2,{\dots},\omega\},$ $\alpha\geq0,$ $d$ is a nonnegative
integer and $\sum_{n=1}^{\omega}z_i^{-2}<\infty.$ If $\omega=0,$
then, by convention, the product is defined to be $1.$ Now, recall the following result
(for more details we refer to Skovgaard's paper \cite{skov}): if a real entire function $\phi$ belongs
to the Laguerre-P\'olya class then satisfies the Laguerre
type inequalities
$$\left[\phi^{(m)}(z)\right]^2-\phi^{(m-1)}(z)\phi^{(m+1)}(z)\geq0,
$$
for $m\in\{1,2,\dots\}$ and all admissible values of $z.$ The infinite product representation \eqref{product},
and the result of Steinig \cite{steinig} concerning the fact that the zeros of the Struve function
$\mathbf{H}_{\nu}$ are real when $|\nu|\leq \frac{1}{2},$ show that the function $\mathcal{H}_{\nu}$ belongs
to the Laguerre-P\'olya class, since exponential factors in the product are cancelled due to the symmetry
of the zeros with respect to the origin. This in turn implies that the following Laguerre type inequality
is valid for all $|\nu|\leq \frac{1}{2},$ $x\in\mathbb{R}$ and $m\in\{1,2,\dots\}$
$$\left[\mathcal{H}_{\nu}^{(m)}(x)\right]^2-\mathcal{H}_{\nu}^{(m-1)}(x)\mathcal{H}_{\nu}^{(m+1)}(x)\geq0.
$$
For $m=1$ the above inequality reduces to \eqref{laguerre}.

{\bf E.} Finally, we note that from Hadamard theorem \cite[p. 26]{lev} actually we get that the infinite
product in \eqref{product} is absolutely convergent on compact subsets of the complex plane. This in
turn implies that if we take the purely imaginary number $\mathrm{i}x$ instead of $x$ in \eqref{mittag}, then for
$|\nu|\leq \frac{1}{2}$ and $x\in\mathbb{R}$ we get the new Mittag-Leffler expansion
\begin{equation}\label{mittag2}
\frac{\mathbf{L}_{\nu-1}(x)}{\mathbf{L}_{\nu}(x)}=\frac{2\nu+1}{x}+\sum_{n\geq1}\frac{2x}{x^2+h_{\nu,n}^2},
\end{equation}
where
$$\mathbf{L}_{\nu}(x)=-\mathrm{i}e^{-\frac{1}{2}\pi\mathrm{i}\nu}\mathbf{H}_{\nu}(\mathrm{i}x)=
\left(\frac{x}{2}\right)^{\nu+1}\sum_{n\geq0}\frac{\left(\frac{x}{2}\right)^{2n}}{\Gamma\left(n+\frac{3}{2}\right)
\Gamma\left(n+\nu+\frac{3}{2}\right)}
$$
stands for the modified Struve function of the first kind (see \cite[p. 288]{nist}). Moreover, following
the lines of the proof of part {\bf a} of Theorem \ref{th1} we obtain that
$$\frac{\mathbf{L}_{\nu-1}^2(x)-\mathbf{L}_{\nu-2}(x)\mathbf{L}_{\nu}(x)}{\mathbf{L}_{\nu}^2(x)}=
\frac{1}{x}\frac{\mathbf{L}_{\nu-1}(x)}{\mathbf{L}_{\nu}(x)}
-\left[\frac{\mathbf{L}_{\nu-1}(x)}{\mathbf{L}_{\nu}(x)}\right]'=
\frac{2(2\nu+1)}{x^2}+\sum_{n\geq 1}\frac{4x^2}{(x^2+h_{\nu,n}^2)^2}\geq0
$$
for all $|\nu|\leq\frac{1}{2}$ and $x\neq0.$ Thus, we obtained the Tur\'an type inequality
$$\mathbf{L}_{\nu}^2(x)-\mathbf{L}_{\nu-1}(x)\mathbf{L}_{\nu+1}(x)\geq0,
$$
where $\nu\in\left[-\frac{3}{2},-\frac{1}{2}\right]$ and $x\in\mathbb{R}.$ For $\nu>-\frac{3}{2}$ and
$x>0$ this Tur\'an type inequality was proved also in \cite{bp3,joshi} by using completely different methods.

\section{\bf Tur\'an type inequalities for Struve functions of the second kind}
\setcounter{equation}{0}

In this section we concentrate on the Struve function of the second kind $\mathbf{K}_{\nu},$ which for
$\nu>-\frac{1}{2}$ has the integral representation \cite[p. 292]{nist}
\begin{equation}\label{secondkind}
\mathbf{K}_{\nu}(x)=\mathbf{H}_{\nu}(x)-Y_{\nu}(x)=\frac{2(\frac{x}{2})^\nu}{\sqrt{\pi}\Gamma(\nu+\frac{1}{2})}
\int_0^{\infty}(1+t^2)^{\nu-\frac{1}{2}}e^{-xt}dt.
\end{equation}
Here $Y_{\nu}$ stands for the Bessel function of the second kind of order $\nu.$ Now, consider the
function $\mathcal{K}_{\nu}:\mathbb{R}\to(0,\infty),$ defined by $\mathcal{K}_{\nu}(x)
=2^{\nu}x^{-\nu}\Gamma\left(\nu+\frac{1}{2}\right)\mathbf{K}_{\nu}(x).$ The next result is the
counterpart of the similar results for modified Struve functions of the first and second kind, see \cite{bp2,bp3}.

\begin{theorem}\label{th2}
The following assertions are true:
\begin{enumerate}
\item[\bf a.] The function $x\mapsto \mathcal{K}_{\nu}(x)$ is completely monotonic and log-convex on $(0,\infty)$
for all $\nu>-\frac{1}{2}.$ Moreover, the following inequality is valid for $\nu>-\frac{1}{2}$ and $x>0$
\begin{equation}\label{E1}
\frac{x\mathbf{K}'_{\nu}(x)}{\mathbf{K}_{\nu}(x)}<\nu.
\end{equation}
\item[\bf b.] The function $\nu\mapsto \mathcal{K}_{\nu}(x)$ is completely monotonic and log-convex on
 $(-\frac{1}{2},\infty)$ for all $x>0.$
Moreover, the following Tur\'an type inequality is valid for all $x>0$ and $\nu>\frac{1}{2}$
\begin{equation}\label{T1}
\mathbf{K}_{\nu}^2(x)-\mathbf{K}_{\nu-1}(x)\mathbf{K}_{\nu+1}(x)\leq\frac{1}{\nu+\frac{1}{2}}\mathbf{K}_{\nu}^2(x)
\end{equation}
\item[\bf c.] The function $x\mapsto \mathbf{K}_{\nu}(x)$ is completely monotonic and log-convex on $(0,\infty)$
for all $\nu\in\left(-\frac{1}{2},0\right].$
\item[\bf d.] The function $x\mapsto x\mathbf{K}_{\nu}'(x)/\mathbf{K}_{\nu}(x)$ is increasing on
$(0,\infty)$ for all $\nu>\frac{1}{2}.$ Moreover, the following Tur\'an type inequality holds for $\nu>-\frac{1}{2}$ and $x>0$
\begin{equation}\label{turanK}
\mathbf{K}_{\nu}^2(x)-\mathbf{K}_{\nu-1}(x)\mathbf{K}_{\nu+1}(x)\leq\frac{2}{x}\mathbf{K}_{\nu}(x)\mathbf{K}_{\nu+1}(x).
\end{equation}
\item[\bf e.] For all $x>0$ and $\nu\in\left(-\frac{1}{2},0\right)$ we have
\begin{equation}\label{R1}
\mathcal{K}_{\nu}(x)<\frac{\Gamma(-\nu)}{\Gamma(\frac{1}{2}-\nu)}.
\end{equation}
\item[\bf f.] For all $x,y>0$ and $\nu\in\left(-\frac{1}{2},0\right)$ we have
\begin{equation}\label{R2}
\mathcal{K}_{\nu}(x+y)\geq\frac{\Gamma(\frac{1}{2}-\nu)}{\Gamma(-\nu)}\mathcal{K}_{\nu}(x)\mathcal{K}_{\nu}(y).
\end{equation}
\item[\bf g.] For all $\nu>\frac{3}{2}$ and $x>0$ we have
\begin{equation}\label{R3}
\mathcal{K}_{\nu-1}(x)\mathcal{K}_{\nu+1}(x)<\mathcal{K}_{\frac{1}{2}}(x)\mathcal{K}_{2\nu-\frac{1}{2}}(x)
\end{equation}
and the reverse inequality holds when $\nu\in\left(\frac{1}{2},\frac{3}{2}\right)$ and $x>0$.
\end{enumerate}
\end{theorem}

\begin{proof}[\bf Proof]
{\bf a.} \& {\bf b.} By using \eqref{secondkind} we obtain that for $\nu>-\frac{1}{2}$ the next integral
representation is valid
\begin{equation}\label{integK}
\mathcal{K}_{\nu}(x)=2^{\nu}x^{-\nu}\Gamma\left(\nu+\frac{1}{2}\right)\mathbf{K}_{\nu}(x)=\frac{2}{\sqrt{\pi}}
\int_0^{\infty}(1+t^2)^{\nu-\frac{1}{2}}e^{-xt}dt.
\end{equation}
Consequently, for $n,m\in \{0,1,2,\ldots \}$ and $\nu>-\frac{1}{2}$ we have
$$(-1)^n[\mathcal{K}_{\nu}(x)]^{(n)}=\frac{2}{\sqrt{\pi}}
\int_0^{\infty}t^n(1+t^2)^{\nu-\frac{1}{2}}e^{-xt}dt,
$$
and
$$(-1)^m\frac{\partial^m\mathcal{K}_{\nu}(x)}{\partial \nu^m}=\frac{2}{\sqrt{\pi}}
\int_0^{\infty}\left(\log\frac{1}{1+t^2}\right)^m(1+t^2)^{\nu-\frac{1}{2}}e^{-xt}dt.
$$
Therefore the functions $x\mapsto \mathcal{K}_{\nu}(x)$ and $\nu\mapsto \mathcal{K}_{\nu}(x)$ are completely
monotonic and hence are
log-convex, since every completely monotonic function is log convex (see \cite[p. 167]{widder}).
Alternatively, the log-convexity of these functions can be proved also by using the H\"older-Rogers inequality
for integrals. Complete monotonicity and log-convexity of $\mathcal{K}_{\nu}$ can be concluded also by noticing
that this function is in fact a Laplace transform.

Now, to prove the Tur\'an type inequality \eqref{T1}, note that $\nu\mapsto \mathcal{K}_{\nu}(x)$ is
log convex on $(-\frac{1}{2},\infty)$ for all $x>0$, which implies that for all $\nu_{1},\nu_{2}>-\frac{1}{2},$
$\alpha\in[0,1]$ and $x>0$ we have
$$\mathcal{K}_{\alpha\nu_{1}+(1-\alpha)\nu_{2}}(x)\leq \left[\mathcal{K}_{\nu_{1}}(x)\right]^{\alpha}
\left[\mathcal{K}_{\nu_{2}}(x)\right]^{1-\alpha}.
$$
Choosing $\nu_{1}=\nu-1,$ $\nu_{2}=\nu+1,$ and $\alpha=\frac{1}{2}$, the above inequality reduces to the Tur\'an type inequality
$$\mathcal{K}_{\nu}^2(x)-\mathcal{K}_{\nu-1}(x)\mathcal{K}_{\nu+1}(x)\leq0,
$$
which is equivalent to the inequality \eqref{T1}.

Alternatively, \eqref{T1} can be proved also as follows. For this let us consider the notation
$$\Theta_{\nu}(x)=\mathbf{K}_{\nu}^2(x)-\mathbf{K}_{\nu-1}(x)\mathbf{K}_{\nu+1}(x).
$$
By using the integral representation \eqref{secondkind} we get
\begin{align*}
\Theta_{\nu}(x)&=\frac{4}{\pi}\left(\frac{x}{2}\right)^{2\nu}
\int_0^{\infty}\int_0^{\infty}e^{-x(t+s)}(1+t^2)^{\nu-\frac{1}{2}}(1+s^2)^{\nu-\frac{3}{2}}
\left[\frac{1+s^2}{\Gamma^2\left(\nu+\frac{1}{2}\right)}-\frac{1+t^2}{\Gamma\left(\nu-\frac{1}{2}\right)\Gamma\left(\nu+\frac{3}{2}\right)}\right]dtds\\
&=\frac{4}{\pi}\left(\frac{x}{2}\right)^{2\nu}\int_0^{\infty}\int_0^{\infty}e^{-x(t+s)}(1+s^2)^{\nu-\frac{1}{2}}(1+t^2)^{\nu-\frac{3}{2}}
\left[\frac{1+t^2}{\Gamma^2\left(\nu+\frac{1}{2}\right)}-\frac{1+s^2}{\Gamma\left(\nu-\frac{1}{2}\right)\Gamma\left(\nu+\frac{3}{2}\right)}\right]dtds\\
&=\frac{2}{\pi}\left(\frac{x}{2}\right)^{2\nu}\int_0^{\infty}\int_0^{\infty}e^{-x(t+s)}(1+t^2)^{\nu-\frac{3}{2}}(1+s^2)^{\nu-\frac{3}{2}}\cdot \mathbf{E}_{\nu}(t,s)dtds,
\end{align*}
where
\begin{align*}\mathbf{E}_{\nu}(t,s)&=(1+t^2)\left[\frac{1+s^2}{\Gamma^2\left(\nu+\frac{1}{2}\right)}-\frac{1+t^2}
{\Gamma\left(\nu-\frac{1}{2}\right)\Gamma\left(\nu+\frac{3}{2}\right)}\right]+(1+s^2)\left[\frac{1+t^2}{\Gamma^2\left(\nu+\frac{1}{2}\right)}-\frac{1+s^2}
{\Gamma\left(\nu-\frac{1}{2}\right)\Gamma\left(\nu+\frac{3}{2}\right)}\right]\\
&=\frac{2(1+t^2)(1+s^2)}{\Gamma^2\left(\nu+\frac{1}{2}\right)}-\frac{(1+t^2)^2}
{\Gamma\left(\nu-\frac{1}{2}\right)\Gamma\left(\nu+\frac{3}{2}\right)}-\frac{(1+s^2)^2}
{\Gamma\left(\nu-\frac{1}{2}\right)\Gamma\left(\nu+\frac{3}{2}\right)}\\
&=\frac{1}{\Gamma\left(\nu+\frac{1}{2}\right)\Gamma\left(\nu+\frac{3}{2}\right)}\mathbf{d}_{\nu}(t,s)\end{align*}
and
$$\mathbf{d}_{\nu}(t,s)=2\left(\nu+\frac{1}{2}\right)(1+t^2)(1+s^2)
-\left(\nu-\frac{1}{2}\right)(1+t^2)^2-\left(\nu-\frac{1}{2}\right)(1+s^2)^2.
$$
Since
$$\mathbf{d}_{\nu}(t,s)\leq2(1+t^2)(1+s^2) \ \ \Longleftrightarrow \ \ -\left(\nu-\frac{1}{2}\right)(t^2-s^2)^2\leq0
$$
for all $t,s\geq0$ and $\nu>\frac{1}{2},$ it follows that
$$\Theta_{\nu}(x)\leq\frac{4}{\pi}\left(\frac{x}{2}\right)^{2\nu}
\frac{1}{\Gamma\left(\nu+\frac{1}{2}\right)\Gamma\left(\nu+\frac{3}{2}\right)}
\int_0^{\infty}\int_0^{\infty}e^{-x(t+s)}(1+t^2)^{\nu-\frac{1}{2}}(1+s^2)^{\nu-\frac{1}{2}}dtds,
$$
which is equivalent to \eqref{T1}.

Now, we prove the inequality \eqref{E1}. Since for $\nu>-\frac{1}{2}$ the function
$x\mapsto x^{-\nu}\mathbf{K}_{\nu}(x)$ is completely monotonic on $(0,\infty),$ in particular it is also decreasing.
Consequently, the function $x\mapsto \log \left(x^{-\nu}\mathbf{K}_{\nu}(x)\right)$ is also decreasing on $(0,\infty)$
for $\nu>-\frac{1}{2}$ and hence $\left(\log \left(x^{-\nu}\mathbf{K}_{\nu}(x)\right)\right)'<0,$
which in turn implies \eqref{E1}.

{\bf c.} By definition of $\mathcal{K}_{\nu}(x)$ we have
$$\mathbf{K}_{\nu}(x)=\frac{x^{\nu}\mathcal{K}_{\nu}(x)}{2^{\nu}\Gamma\left(\nu+\frac{1}{2}\right)}.
$$
Note that $x\mapsto x^{\nu}$ is completely monotonic on $(0,\infty)$ for all $\nu\leq0$. Thus by part {\bf a} of
this theorem the function $x\mapsto \mathbf{K}_{\nu}(x)$, as a product of two completely monotonic functions, is
completely monotonic and hence log-convex on $(0,\infty)$ for all $\nu\in\left(-\frac{1}{2},0\right].$

{\bf d.} By using $\mathbf{K}_{\nu}(x)=\mathbf{H}_{\nu}(x)-Y_{\nu}(x)$ and the corresponding recurrence
relations for $\mathbf{H}_{\nu}$ and $Y_{\nu}$ we can see that the Struve function of the second kind $\mathbf{K}_{\nu}$
satisfies the same recurrence relations like the Struve function of the first kind $\mathbf{H}_{\nu},$
that is, if we replace $\mathbf{H}_{\nu}$ with $\mathbf{K}_{\nu}$ in \eqref{rec1}, \eqref{rec2},
\eqref{rec3} and \eqref{rec4}, then these recurrence relations remain true. The analogous of \eqref{rec1} is the following
\begin{equation}\label{rec1K}
\mathbf{K}_{\nu-1}(x)=\frac{\nu}{x}\mathbf{K}_{\nu}(x)+\mathbf{K}_{\nu}'(x).
\end{equation}
Now, by using \eqref{integK} and \eqref{rec1K} we obtain
\begin{align*}
&\left[\frac{x\mathbf{K}_{\nu}'(x)}{\mathbf{K}_{\nu}(x)}\right]'=\left[\frac{x\mathbf{K}_{\nu-1}(x)}{\mathbf{K}_{\nu}(x)}\right]'=
(2\nu-1)\left[\left.{\int_0^{\infty}(1+t^2)^{\nu-\frac{3}{2}}e^{-xt}dt}\right/{\int_0^{\infty}(1+t^2)^{\nu-\frac{1}{2}}e^{-xt}dt}\right]'\\
&\ \ \ =(2\nu-1)\left.{\int_0^{\infty}\int_0^{\infty}(1+t^2)^{\nu-\frac{3}{2}}(1+s^2)^{\nu-\frac{1}{2}}e^{-x(t+s)}(s-t)dtds}\right/{\left(\int_0^{\infty}(1+t^2)^{\nu-\frac{1}{2}}e^{-xt}dt\right)^2}\\
&\ \ \ =(2\nu-1)\left.{\int_0^{\infty}\int_0^{\infty}(1+s^2)^{\nu-\frac{3}{2}}(1+t^2)^{\nu-\frac{1}{2}}e^{-x(t+s)}(t-s)dtds}\right/{\left(\int_0^{\infty}(1+t^2)^{\nu-\frac{1}{2}}e^{-xt}dt\right)^2}\\
&\ \ \ =\left(\nu-\frac{1}{2}\right)\left.{\int_0^{\infty}\int_0^{\infty}((1+t^2)(1+s^2))^{\nu-\frac{3}{2}}e^{-x(t+s)}(t-s)^2(t+s)dtds}
\right/{\left(\int_0^{\infty}(1+t^2)^{\nu-\frac{1}{2}}e^{-xt}dt\right)^2}.
\end{align*}
Thus, indeed the function $x\mapsto x\mathbf{K}_{\nu}'(x)/\mathbf{K}_{\nu}(x)$ is increasing on $(0,\infty)$ for all $\nu>\frac{1}{2}.$
Now, appealing to the above result and to the recurrence relation \eqref{rec1K} we obtain for $x>0$ and $\nu>\frac{1}{2}$
$$0\leq\left[\frac{x\mathbf{K}_{\nu-1}(x)}{\mathbf{K}_{\nu}(x)}\right]'
=2\frac{\mathbf{K}_{\nu-1}(x)}{\mathbf{K}_{\nu}(x)}-\frac{x\Theta_{\nu-1}(x)}{\mathbf{K}_{\nu-1}(x)}.
$$
Changing $\nu$ with $\nu+1$ we get \eqref{turanK}. Alternatively, \eqref{turanK} can be proved by using
$$0\geq\left[\frac{\mathbf{K}_{\nu+1}(x)}{x\mathbf{K}_{\nu}(x)}\right]'=
\frac{\Theta_{\nu}(x)}{x\mathbf{K}_{\nu}^2(x)}-\frac{2}{x^2}\frac{\mathbf{K}_{\nu+1}(x)}{\mathbf{K}_{\nu}(x)},
$$
where $x>0$ and $\nu>-\frac{1}{2}.$

{\bf e.} By part {\bf a} of this theorem, the function $\mathcal{K}_{\nu}$ is decreasing on $(0,\infty)$
for all $\nu>-\frac{1}{2},$ and hence we get
$$\mathcal{K}_{\nu}(x)<\frac{2}{\sqrt{\pi}}\int_0^{\infty}(1+t^2)^{\nu-\frac{1}{2}}dt.
$$
Now, by using \cite[p. 142]{nist}
$$B(a,b)=\frac{\Gamma(a)\Gamma(b)}{\Gamma(a+b)}=\int_0^{\infty}\frac{t^{a-1}dt}{(1+t)^{a+b}}
=\int_0^{\infty}\frac{2u^{2a-1}du}{(1+u^2)^{a+b}}
$$
it can be shown that for $\nu<0$ we have
$$\frac{2}{\sqrt{\pi}}\int_0^{\infty}(1+t^2)^{\nu-\frac{1}{2}}dt=\frac{\Gamma(-\nu)}{\Gamma(\frac{1}{2}-\nu)}.
$$
Consequently, for all $x>0$ and $\nu\in\left(-\frac{1}{2},0\right)$ we obtain the inequality \eqref{R1}.

{\bf f.} From inequality \eqref{R1} and part {\bf a} of this theorem we have that the function
$$x\mapsto \frac{\Gamma(\frac{1}{2}-\nu)}{\Gamma(-\nu)}\mathcal{K}_{\nu}(x)
$$
maps $(0,\infty)$ into $(0,1)$ and it is completely monotonic
on $(0,\infty)$ for all $\nu\in\left(-\frac{1}{2},0\right)$.
Now, recall the result of Kimberling \cite{kimberling}, which says that if a function $f$, defined on
$(0,\infty)$, is continuous and completely monotonic and maps $(0,\infty)$ into $(0,1)$, then $\log f$
is super-additive, that is, for all
$x,y>0$ we have
$$ \log f(x+y)\geq\log f(x)+\log f(y)
$$
or
$$f(x+y)\geq f(x)f(y).
$$
Applying this result the inequality \eqref{R2} follows.

{\bf g.} To prove the inequality \eqref{R3} we use the Chebyshev integral inequality \cite[p. 40]{mitrinovic}:
Let $f$ and $g$ be functions which are integrable and monotone in the same sense (i.e.
either both increasing or both decreasing) on $(a,b)$ and let $p$ be a positive and integrable
function on $(a,b)$. Then
\begin{equation}\label{chebyshev}
\int_a^{b}p(t)f(t)dt\int_a^{b}p(t)g(t)dt\leq\int_a^{b}p(t)dt\int_a^{b}p(t)f(t)g(t)dt.
\end{equation}
If $f$ and $g$ are monotone in opposite sense (i.e. one is decreasing and the other is increasing), then
the inequality \eqref{chebyshev} is reversed.

Now let $p,f$ and $g$ be functions defined on $(0,\infty)$ such that
$$p(t)=e^{-xt}, ~  f(t)=\frac{2}{\sqrt{\pi}}(1+t^2)^{\nu-\frac{3}{2}}~ \mbox{ and }~
g(t)=\frac{2}{\sqrt{\pi}}(1+t^2)^{\nu+\frac{1}{2}}.
$$
Note that $f$ is increasing if $\nu>\frac{3}{2},$ and $g$ is increasing if $\nu>-\frac{1}{2}.$
Now substituting $p,f$ and $g$ in \eqref{chebyshev} and using the following
$$\mathbf{H}_{\frac{1}{2}}(x)=\sqrt{\frac{2}{\pi x}}(1-\cos x), ~
Y_{\frac{1}{2}}(x)=-\sqrt{\frac{2}{\pi x}}\cos x
$$
and the definition of $\mathcal{K}_{\nu}(x)$ we have the desired inequality  \eqref{R3}
valid for all $\nu>\frac{3}{2}$ and $x>0.$ As $f$ is decreasing for $\nu<\frac{3}{2},$
the inequality \eqref{R3} is reversed when $\nu\in\left(\frac{1}{2},\frac{3}{2}\right).$
\end{proof}

\subsection*{Acknowledgments} The work of \'A. Baricz was supported by the J\'anos Bolyai Research Scholarship of
the Hungarian Academy of Sciences. The second author is on leave from the Department of Mathematics,
Indian Institute of Technology Madras, Chennai-600 036, India. The research of S. Singh was supported by the
fellowship of the University Grants Commission, India.

\end{document}